\documentclass[12pt,oneside,english]{amsart}
\usepackage[utf8]{luainputenc}
\usepackage{amsthm}
\usepackage{amssymb}

\makeatletter
\numberwithin{equation}{section}
\numberwithin{figure}{section}
  \theoremstyle{plain}
  \newtheorem*{conjecture*}{\protect\conjecturename}
\theoremstyle{plain}
\newtheorem{thm}{\protect\theoremname}
  \theoremstyle{plain}
  \newtheorem{prop}[thm]{\protect\propositionname}
  \theoremstyle{plain}
  \newtheorem{conjecture}[thm]{\protect\conjecturename}
  \theoremstyle{plain}
  \newtheorem{cor}[thm]{\protect\corollaryname}
  \theoremstyle{plain}
  \newtheorem{lem}[thm]{\protect\lemmaname}
  \theoremstyle{remark}
  \newtheorem{rem}[thm]{\protect\remarkname}

\usepackage{amsmath}
\usepackage{geometry}
\usepackage{tikz-cd}

\makeatother

\usepackage{babel}
  \providecommand{\conjecturename}{Conjecture}
  \providecommand{\corollaryname}{Corollary}
  \providecommand{\lemmaname}{Lemma}
  \providecommand{\propositionname}{Proposition}
  \providecommand{\remarkname}{Remark}
\providecommand{\theoremname}{Theorem}

\begin{document}

\title{REFLECTION PRINCIPLES FOR CLASS GROUPS}

\author{JACK KLYS}
\begin{abstract}
We present several new examples of reflection principles which apply
to both class groups of number fields and picard groups of of curves
over $\mathbb{P}^{1}/\mathbb{F}_{p}$. This proves a conjecture of
Lemmermeyer \cite{franz1} about equality of 2-rank in subfields of
$A_{4}$, up to a constant not depending on the discriminant in the
number field case, and exactly in the function field case. More generally
we prove similar relations for subfields of a Galois extension with
group $G$ for the cases when $G$ is $S_{3}$, $S_{4}$, $A_{4}$,
$D_{2l}$ and $\mathbb{Z}/l\mathbb{Z}\rtimes\mathbb{Z}/r\mathbb{Z}$.
The method of proof uses sheaf cohomology on 1-dimensional schemes,
which reduces to Galois module computations. 
\end{abstract}

\maketitle

\section*{Introduction}

In this paper we look at the problem of relating the size of the $l$-torsion
in the class groups of two distinct number fields. We let $\mbox{rk}_{l}Cl\left(K\right)=\mbox{dim}_{\mathbb{F}_{l}}Cl\left(K\right)\left[l\right]$. 

Describing the size of the $l$-torsion of the class group of a number
field is in general a hard problem. There are special cases where
one can say something about $\mbox{rk}_{l}Cl\left(K\right)$. For
example it is easy to show for any quadratic field $K$ that $\mbox{rk}_{2}Cl\left(K\right)=r-1$,
where $r$ is the number of primes ramified in $K$. One source of
theorems describing $l$-torsion in class groups is Iwasawa theory,
which gives formulas for $\mbox{rk}_{l}Cl\left(K\right)$ for $K$
lying in some tower of number fields, in terms of certain invariants
depending on the base field. A very strong asymptotic conjecture on
class group order is the following due to Zhang \cite{Zhang}:
\begin{conjecture*}
For any number field $K$ let $n=\left[K:\mathbb{Q}\right]$ and let
$\epsilon>0$. Then 
\[
\left|Cl\left(K\right)\left[l\right]\right|\ll_{\epsilon,n,l}D_{K}^{\epsilon}.
\]
 
\end{conjecture*}
That is the $l$-torsion in number fields of a fixed degree grows
slower than any power of their discriminant. Some work has been done
in this direction by Ellenberg and Venkatesh \cite{EllenbergVenkatesh}
by combining reflection principles with analytic techniques.

A slightly different kind of problem is that of relating $l$-rank
in class groups of two different number fields. Such statements are
often called reflection principles. There are many such statements
known and we will give a very brief overview of some of these. For
a more exhaustive exposition we refer the reader to \cite{franz1}.

The following is refered to as the Scholz reflection theorem \cite{Washington199612}: 
\begin{prop}
Let $D>1$ be square-free. Let $K=\mathbb{Q}(\sqrt{-3D})$ and $F=\mathbb{Q}(\sqrt{D})$.
Then $\mathrm{rk}_{l}Cl(F)\le\mathrm{rk}_{l}Cl(K)\le\mathrm{rk}_{l}Cl(F)+1$. 
\end{prop}
And a generalization of the above due to Leopoldt \cite{Washington199612}:
\begin{prop}
Let $K=\mathbb{Q}\left(\zeta_{l}\right)$. Let $Cl\left(K\right)^{\pm}$
denote the positive and negative eigenspace of $Cl\left(K\right)$
under the action of complex conjugation in $\mathrm{Gal}\left(K/\mathbb{Q}\right)$.
Then $\mathrm{rk}_{l}Cl\left(K\right)^{+}\le\mathrm{rk}_{l}Cl\left(K\right)^{-}$.
\end{prop}
Another result of this form is due to Bolling for subfields of dihedral
extensions \cite{bolling}:
\begin{prop}
\label{prop:Bolling}Let $L$ be a number field with $\mathrm{Gal}\left(L/\mathbb{Q}\right)=D_{2l}$,
where p is an odd prime, and let $K$ be any of its subfields of degree
p. Assume that the quadratic subfield F of L is complex. Then 
\[
\mathrm{rk}_{l}Cl\left(F\right)-1\le\mathrm{rk}_{l}Cl\left(K\right)\le\frac{\left(l-1\right)}{2}\left(\mathrm{rk}_{l}Cl\left(F\right)-1\right).
\]

\end{prop}
Recently Tsimerman proved a result about class groups of agebraic
tori \cite{jacob}, which he applied to derive reflection principles
similar to the above. It has the downside that the error term depends
on the discriminant of the field:
\begin{prop}
\label{prop:jacob's a4}Let $L$ be a number field with $\mathrm{Gal}\left(L/\mathbb{Q}\right)=S_{4}$
or $A_{4}$. Let $K_{1}$ be a quartic subfield, and $K_{2}$ be its
cubic resolvent. Then $\mathrm{rk}_{2}Cl\left(K_{1}\right)=\mathrm{rk}_{2}Cl\left(K_{2}\right)+O_{\epsilon}\left(D_{L}^{\epsilon}\right)$.
\end{prop}
Lemmermeyer conjectures \cite{franz1} that in this case the following
(sharp) bound holds:
\begin{conjecture}
\label{conj}In the above $\mathrm{rk}_{2}Cl\left(K_{1}\right)-2\le\mathrm{rk}_{2}Cl\left(K_{2}\right)\le\mathrm{rk}_{2}Cl\left(K_{1}\right)$.
\end{conjecture}
We will build on Tsimerman's approach and restate the problem using
sheaf cohomology. This lets us prove reflection principles in the
general setting of picard groups of 1-dimensional schemes, which can
then be applied to derive reflection principles for both class groups
of number fields and picard groups of function fields. Among the results
are an improvement of Proposition \ref{prop:jacob's a4} by removing
the dependence on the discriminant in favor of a bound similar to
that in the other statements above, which proves Conjecture \ref{conj}
up to an error of $O\left(1\right)$. 

We summarize our results below, and state them in more detail in the
main part of the paper. 
\begin{thm}
\label{firstthm}Let $S$ be either $\mathrm{Spec}\mathbb{Z}\left[1/l\right]$
or $\mathbb{P}^{1}/\mathbb{F}_{p}$ where $p$ and $l$ are distinct
primes and let $\eta=\mbox{spec}F$ be the generic point of $S$.
Consider two finite covers $\pi_{i}:X_{i}\longrightarrow S$ for $i=1,2$
with generic points $\mathrm{Spec}K_{i}$. Let $L/F$ be a Galois
extension containing the $K_{i}$ and $G=\mathrm{Gal}\left(L/F\right)$.
We will always make the assumption that $L/K_{2}$ is unramified.
Suppose we have one of the following cases:
\begin{enumerate}
\item Let $l=3$. Let $G=S_{3}$. Let $K_{1}$ be a cubic subfield and $K_{2}$
its quadratic resolvent.
\item Let $l=2$. Let $G=A_{4}$ or $S_{4}$. Let $K_{1}$ be a quartic
subfield and $K_{2}$ its cubic resolvent.
\end{enumerate}
Then we have the bound
\[
\mathrm{rk}_{l}Cl\left(K_{2}\right)-C_{1}\le\mathrm{rk}_{l}Cl\left(K_{1}\right)\le\mathrm{rk}_{l}Cl\left(K_{2}\right)+C_{2}
\]
 where the $C_{i}$ are constants depending on the case and on whether
$\mu_{l}\subset\mathbb{F}_{p}$ if $S=\mathbb{P}^{1}$, and can be
computed explicitly.
\end{thm}
Computing the above constants explicitly yields the following Corollary
which proves Conjecture \ref{conj} up to an error of $O\left(1\right)$.
\begin{cor}
Let $L/\mathbb{Q}$ be a Galois extension with $\mathrm{Gal}\left(L/\mathbb{Q}\right)=A_{4}$.
Let $K_{1}$ be a quartic subfield and $K_{2}$ its cubic resolvent.
Then if $L$ is real we have 
\[
\mathrm{rk}_{2}Cl\left(K_{2}\right)-10\le\mathrm{rk}_{2}Cl\left(K_{1}\right)\le\mathrm{rk}_{2}Cl\left(K_{2}\right)+10
\]
 and if $L$ is complex we have 
\[
\mathrm{rk}_{2}Cl\left(K_{2}\right)-8\le\mathrm{rk}_{2}Cl\left(K_{1}\right)\le\mathrm{rk}_{2}Cl\left(K_{2}\right)+12.
\]
 
\end{cor}
Additionally in the function field case we obtain Conjecture \ref{conj}
exactly:
\begin{cor}
Let $L/\mathbb{F}_{p}\left(T\right)$ be a Galois extension with $\mathrm{Gal}\left(L/\mathbb{F}_{p}\left(T\right)\right)=A_{4}$.
Let $K_{1}$ be a quartic subfield and $K_{2}$ its cubic resolvent.
Then
\[
\mathrm{rk}_{2}\mathrm{Pic}C_{2}-2\le\mathrm{rk}_{2}\mathrm{Pic}C_{1}\le\mathrm{rk}_{2}\mathrm{Pic}C_{2}.
\]

\end{cor}
We also obtain reflection principles in a new case, when $\mathrm{Gal}\left(L/\mathbb{Q}\right)=\mathbb{Z}/l\mathbb{Z}\rtimes\mathbb{Z}/r\mathbb{Z}$,
which can be thought of as a generalization of the dihedral case of
Proposition \ref{prop:Bolling}. Let the notation be the same as in
Theorem \ref{firstthm}.
\begin{thm}
With the same assumptions as in Theorem \ref{firstthm}, suppose we
have one of the following cases:
\begin{enumerate}
\item Let $l$ be an odd prime. Let $G=D_{2l}$. Let $K_{1}$ be any of
the subfields of degree $l$ and $K_{2}$ be the quadratic subfield.
\item Let $l,r$ be odd primes with $r\equiv1\left(l\right)$. Let $G=\mathbb{Z}/l\mathbb{Z}\rtimes\mathbb{Z}/r\mathbb{Z}$.
Let $K_{1}$ be any of its subfields of degree $l$ and $K_{2}$ be
the subfield of degree $r$.
\end{enumerate}
Then we have the bound
\[
C_{1}\mathrm{rk}_{l}Cl\left(K_{2}\right)+C_{2}\le\mathrm{rk}_{l}Cl\left(K_{1}\right)\le C_{3}\mathrm{rk}_{l}Cl\left(K_{2}\right)+C_{4}
\]
 where the $C_{i}$ are constants depending on the case and on whether
$\mu_{l}\subset\mathbb{F}_{p}$ if $S=\mathbb{P}^{1}$, and can be
computed explicitly.
\end{thm}
We emphasize that the constants do not depend on the discriminant
of the field. The constants differ in each case and between the number
field and function field setting. We will consider each case separately
and compute them explicitly.

We also note that even though $S_{3}$ specializes $D_{2l}$ which
specializes $\mathbb{Z}/l\mathbb{Z}\rtimes\mathbb{Z}/r\mathbb{Z}$,
we state the cases separately since our constants improve with each
specialization.

\subsection*{Comments and further directions:}

It is likely that the bounds we obtain can be improved by a more careful
consideration of the morphisms in our long exact sequences. For example
in Lemma \ref{lem:S3functionfield} determining $s$ and $t$ would
require computing the maps $H^{1}\left(M'\right)\longrightarrow H^{2}\left(\mu_{3}\right)$
and $H^{1}\left(\mu_{3}\right)\longrightarrow H^{1}\left(N'\right)$
respectively. One possible way of doing this is explicitly computing
the maps in terms of Cech cohomology. However this approach can become
computationally tedious.

We note that the bounds we obtain are generally sharper in the function
field case, and especially when $\mu_{l}\notin\mathbb{F}_{p}$. In
the number field case precision is lost as a result of inverting the
prime $l$, since this increases the size of the unit group and forces
us to work with the class group away from the primes above $l$, both
of which play a key role in our computations.

It would be interesting to find more examples of field extensions
for which such reflection principles hold. Our method applies more
generally to any subfields of a Galois extension provided there exists
a series of exact sequences relating the two modules generated by
the embeddings of each subfield, along with an additional condition.
The procedure we have used for computing such sequences is to find
the Jordan-Holder decomposition of each module. By relating the modules
we mean roughly that the composition factors in their Jordan-Holder
decompositions are the same. The additional condition is that the
sequences remain exact upon taking invariants by certain subgroups
of the Galois group (we refer to the remainder of the paper for details
regarding this). It would be interesting to gain a better understanding
of when this type of situation occurs.

\section{Preliminary Results\label{sec:Preliminary-Results}}

We start by defining the notation and developing some basic results
which will be used throughout each of the examples. By $H^{i}$ we
will always mean $H_{et}^{i}$.

\subsection{Schemes and Picard groups}

Let $l$ be a prime. Let $S$ be the scheme equal to either $\mathrm{Spec}\mathbb{Z}\left[1/l\right]$
or $\mathbb{P}_{\mathbb{F}_{p}}^{1}$ where $p$ is a prime distinct
from $l$. Let $\eta=\mbox{Spec}F$ be the generic point of $S$ and
denote by $g:\eta\longrightarrow S$ the inclusion. Consider a scheme
$\pi:X\longrightarrow S$ which is a finite degree $n$ cover of $S$.
When $S=\mathbb{P}_{\mathbb{F}_{p}}^{1}$ we will let $X$ be a complete
connected smooth curve and when $S=\mathrm{Spec}\mathbb{Z}\left[1/l\right]$
we let $X$ be of the form $\mathrm{Spec}\mathcal{O}_{K}\left[1/l\right]$
for some number field $K$. Let $Z\subset S$ be the finite set of
points above which $\pi$ is ramified. Let $\mbox{Spec}K$ be the
generic point of $X$, so $\left[K:F\right]=n$. 

Consider the Kummer sequence of etale sheaves on $X$

\[
\begin{tikzcd}[row sep=1em]&\\ 1\arrow{r} & \mu_{l} \arrow{r}& \mathbb{G}_{m} \arrow{r} & \mathbb{G}_{m} \arrow{r}&1\\ &
\end{tikzcd}\]We assume that $X$ has residue characteristic coprime to $l$ at
each point, so that the Kummer sequence is exact. Taking cohomology
of this sequence and noting that $H^{1}\left(X,\mathbb{G}_{m}\right)=\mathrm{Pic}\left(X\right)$
gives

\[\begin{tikzcd}[row sep=1em]&\\ 1\arrow{r} & \mathcal{O}_{X}\left(X\right)^{\times}\left[l\right] \arrow{r}& \mathcal{O}_{X}\left(X\right)^{\times} \arrow{r} & \mathcal{O}_{X}\left(X\right)^{\times}\arrow{lld}\\ & H^{1}\left(X,\mu_{l}\right) \arrow{r} & \mathrm{Pic}\left(X\right) \arrow{r} & \mathrm{Pic}\left(X\right) & \\ &\end{tikzcd}\]

Let $\mathcal{L}=\pi_{*}\mu_{l}$. This is a finite locally constant
sheaf on $S\backslash Z$. Since push-forward of sheaves by finite
morphisms is exact and preserves injectives, we have $H^{1}\left(X,\mu_{l}\right)=H^{1}\left(S,\mathcal{L}\right)$.
We summarize this as 
\begin{lem}
\label{lem:pic and h1}Let $\pi:X\longrightarrow S$ be a finite cover
and let $\mathcal{L}=\pi_{*}\mu_{l}$. Then 
\begin{eqnarray*}
\mathrm{Pic}\left(X\right)\left[l\right] & \cong & H^{1}\left(S,\mathcal{L}\right)/\left(\mathcal{O}_{X}\left(X\right)^{\times}/l\mathcal{O}_{X}\left(X\right)^{\times}\right).
\end{eqnarray*}

\end{lem}
In all cases we will consider $\mathcal{O}_{X}\left(X\right)^{\times}/l\mathcal{O}_{X}\left(X\right)^{\times}$
can be computed explicitly and so $H^{1}\left(S,\mathcal{L}\right)$
will be the main object of interest.\\

\subsection{Sheaves and Galois modules}

Next we recall some basics about finite locally constant sheaves.
We continue with the same notation as above. Let $U=S\backslash Z$.
Note $Z$ is closed, so $U$ is open, and let $j:U\longrightarrow S$
be the open immersion. Fix $\overline{\eta}=S\mathrm{pec}\overline{F}$
a geometric point above the generic point $\eta=\mathrm{Spec}F$ of
$S$. As above let $\mathcal{L}=\pi_{*}\mu_{l}$.

Recall that there is a category equivalence between finite $\pi_{1}\left(U,\overline{\eta}\right)$-modules,
finite etale schemes over $U$, and finite locally constant sheaves
of abelian groups on $U$. It says that a finite locally constant
sheaf $\mathcal{F}$ on $U$ is represented by some scheme $Y$ finite
etale over $U$ whose geometric points above $\overline{\eta}$ are
a $\pi_{1}\left(U,\overline{\eta}\right)$-module. 

The following facts are standard and can be found in \cite{Milneetalecohomology}.
\begin{lem}
\label{lem:Let--be-3}Let $M$ be a $G_{F}$-module representing a
sheaf $\mathcal{F}$ on $\eta$. Suppose the action of $G_{F}$ factors
through $\mathrm{Gal}\left(L/F\right)$ for some finite extension
$L$. Let $V\subset S$ be the set of points which are unramified
in $L$. Then $\mathcal{F}$ extends to a finite locally constant
sheaf on V represented by $M$ with an action of $\pi_{1}\left(V,\overline{\eta}\right)$.
\begin{lem}
\label{lem:j_*stalks}Let $\mathcal{F}$ be a finite locally constant
sheaf on $U$ represented by the $\pi_{1}\left(U,\overline{\eta}\right)$-module
$M$. Then for any point $z\in S\backslash U$ we have $\left(j_{*}\mathcal{F}\right)_{\overline{z}}\cong M^{I_{z}}$
as a $D_{z}/I_{z}$-module, where $I_{z}$ is the inertia group at
$z$.
\end{lem}
\end{lem}
\begin{rem}
\label{rem:By-the-above}By the above Lemmas since $\mathcal{L}\mid_{U}$
is finite locally constant it will be represented on $U$ by its stalk
$\mathcal{L}_{\overline{\eta}}$, and since $\mathcal{L}$ is itself
the pushforward of the finite locally constant sheaf $\mu_{l}$, it
can be shown that $j_{*}\left(\mathcal{L}\mid_{U}\right)=\mathcal{L}$.
Thus to describe the stalk $\mathcal{L}_{\overline{z}}$ for any $z\in Z$
we take in $I_{z}$ invariants of $\mathcal{L}_{\overline{\eta}}$.
\end{rem}

Next we use our definition of $\mathcal{L}$ to give a more explicit
description of $M$. Recall that $\left[K:F\right]=n$.
\begin{lem}
\label{lem:Let--be-2}Let $M$ be the $G_{F}$-module representing
the stalk $\mathcal{L}_{\overline{\eta}}$. Let $\sigma_{1},\ldots,\sigma_{n}$
be the embeddings of $K$ into $\overline{F}$. Let $L$ be a finite
extension of $F$ containing the normal closure of $K$ in $\overline{F}$
as well as $\mu_{l}\left(\overline{F}\right)$. Then $M\cong\mathbb{Z}/l\mathbb{Z}\left\langle \sigma_{1},\ldots,\sigma_{n}\right\rangle $
and the action of $G_{F}$ factors through $\mathrm{Gal}\left(L/F\right)$.\end{lem}
\begin{proof}
We claim $\mathcal{L}_{\overline{\eta}}\cong\mathcal{L}\mid_{\eta}\left(\mbox{Spec}L\right)$
as $G_{F}$-modules. By definition $\mathcal{L}\left(\mbox{Spec}L\right)=\mu_{l}\left(L\otimes_{F}K\right)$.
There is an isomorphism $L\otimes_{F}K\cong\prod_{i=1}^{n}L$ since
$L$ contains the Galois closure of $K$. So $\mathcal{L}\left(\mbox{Spec}L\right)=\mu_{l}\left(\mathrm{Spec}\left(\prod_{i=1}^{n}L\right)\right)\cong\prod_{i=1}^{n}\mathbb{Z}/l\mathbb{Z}$,
and furthermore $G_{F}$ acts on $L\otimes_{F}K$ by acting on each
coordinate which under the above isomorphism translates into an action
on each coordinate which also permutes the coordinates in correspondence
with the embeddings of $K$ in $\overline{F}$. Clearly this action
factors through $\mathrm{Gal}\left(L/F\right)$. It is also clear
from this that $\mbox{Spec}L$ trivializes $\mathcal{L}$.
\end{proof}
We now turn our attention towards the main goal of the paper. Consider
two schemes $\pi_{i}:X_{i}\longrightarrow S$ for $i=1,2$ of the
form described at the beginning of this section, with generic point
Spec$K_{i}$. Let $\mathcal{L}_{i}=\left(\pi_{i}\right)_{*}\mu_{l}$.
We want to find a relationship between $\mbox{rk}_{l}\mathrm{Pic}\left(X_{1}\right)$
and $\mbox{rk}_{l}\mathrm{Pic}\left(X_{2}\right)$, and we will do
this by relating $\mbox{rk}_{l}H^{1}\left(S,\mathcal{L}_{1}\right)$
and $\mbox{rk}_{l}H^{1}\left(S,\mathcal{L}_{2}\right)$ and using
Lemma \ref{lem:pic and h1}. The latter will be done by constructing
a family of exact sequences of sheaves on $S$ which contain both
of the $\mathcal{L}_{i}$, as well as other intermediate sheaves.
Taking cohomolgy will then give the desired result. The intermediate
sheaves will depend on the particular example, and we treat each cases
seperately. 

By what we have done so far, the problem is reduced to working with
Galois modules. It only depends on the fields $K_{i}$, and by taking
a suitably large extension $L/F$ containing the normal closure of
each $K_{i}$ we can work with finite $G\left(L/F\right)$-modules.
The strategy is to first find a collection of exact sequences of $G\left(L/F\right)$-modules
containing the $\mathcal{L}_{i}$ - each of these sequences corresponds
to an exact sequence of sheaves at the generic point. Then we fix
a subgroup $I\subset G\left(L/F\right)$ which is the inertia group
of some ramified prime, and take $I$-invariants of the modules. We
require that this preserve exactness, which by Lemma \ref{lem:j_*stalks}
corresponds to exactness of the sheaves at that ramified point. Lemma
\ref{lem:Let--be-2} will give the necessary description of the modules
which are our starting point.

We look at examples in the case of function fields, where $S=\mathbb{P}_{\mathbb{F}_{p}}^{1}$
and $F=\mathbb{F}_{p}\left(T\right)$, and in number fields, where
$S=\mbox{Spec}\mathbb{Z}$ and $F=\mathbb{Q}$. Since the problem
only depends on the Galois theory of the fields in question, any example
has a manifestation in both settings, though the computations vary
in details, such as the computation of the unit group $\mathcal{O}_{X}\left(X\right)^{\times}$,
and hence give different kinds of bounds.

\section{Some Cohomology Computations}

We begin with some lemmas which will be needed in subsequent computations.
We will use the notation $h^{i}\left(X,\mathcal{F}\right)=\log_{l}\left|H^{i}\left(X,\mathcal{F}\right)\right|$
(when the cohomology group has exponent $l$ this is the $l$-rank).

\subsection{Function fields}

We first let $S=\mathbb{P}_{\mathbb{F}_{p}}^{1}$ and $F=\mathbb{F}_{p}\left(T\right)$.
Finite field extensions of $\mathbb{F}_{p}\left(T\right)$ correspond
to curves which are finite covers of $\mathbb{P}^{1}\left(\mathbb{F}_{p}\right)$.
Any such curve $C$ corresponds to its function field $\mathbb{F}_{p}\left(C\right)$. 

We compute the unit groups mod $l$ of a curve $C$ and cohomology
groups of the sheaf $\mu_{l}$ on $S$.
\begin{lem}
\label{lem:unit group mod p functionfield}For any curve C finite
over $S$ we have 
\end{lem}
\[
\mathcal{O}_{C}\left(C\right)^{\times}/\left(\mathcal{O}_{C}\left(C\right)^{\times}\right)^{l}=\begin{cases}
\mathbb{Z}/l\mathbb{Z} & \mbox{if }\mu_{l}\in\mathbb{F}_{p}\\
1 & \mbox{if }\mu_{l}\notin\mathbb{F}_{p}.
\end{cases}
\]

\begin{proof}
Follows since $\mathcal{O}_{C}\left(C\right)^{\times}=\mathbb{F}_{p}^{\times}\cong\mathbb{Z}/\left(p-1\right)\mathbb{Z}$
and $\mu_{l}\in\mathbb{F}_{p}$ is equivalent to $l\mid p-1$.\end{proof}
\begin{lem}
\label{lem:mu_p cohomology}For the sheaf $\mu_{l}$ on S we have

\[
H^{0}\left(\mu_{l}\right),H^{1}\left(\mu_{l}\right),H^{2}\left(\mu_{l}\right)=\begin{cases}
\mathbb{Z}/l\mathbb{Z},\mathbb{Z}/l\mathbb{Z},\mathbb{Z}/l\mathbb{Z} & \mbox{if }\mu_{l}\in\mathbb{F}_{p}\\
0,0,\mathbb{Z}/l\mathbb{Z} & \mbox{if }\mu_{l}\notin\mathbb{F}_{p}.
\end{cases}
\]
\end{lem}
\begin{proof}
We take cohomology of the Kummer sequence on $\mathbb{P}^{1}$ to
get

\[\begin{tikzcd}[row sep=1em]&\\ 
1\arrow{r} & H^{0}\left(\mu_{l}\right) \arrow{r}& \mathbb{F}_{p}^{\times} \arrow{r} & \mathbb{F}_{p}^{\times}\arrow{lld}\\ 
&H^{1}\left(\mu_{l}\right) \arrow{r} & \mathbb{Z} \arrow{r} & \mathbb{Z} \arrow{lld}\\
&H^{2}\left(\mu_{l}\right) \arrow{r} & 0 \arrow{r} & 0\arrow{lld}\\
&H^{3}\left(\mu_{l}\right) \arrow{r} & \mathbb{Q}/\mathbb{Z} \arrow{r} & \mathbb{Q}/\mathbb{Z} \\ & \end{tikzcd}\]since $H^{1}\left(\mathbb{P}^{1},\mathbb{G}_{m}\right)=\mathrm{Pic}\mathbb{P}^{1}=\mathbb{Z}$,
and $H^{2}\left(\mathbb{P}^{1},\mathbb{G}_{m}\right)=0$ and $H^{3}\left(\mathbb{P}^{1},\mathbb{G}_{m}\right)=\mathbb{Q}/\mathbb{Z}$
(see \cite{Milneetalecohomology} p.109).\end{proof}
\begin{lem}
\label{lem:euler char function field}Let $\mathcal{F}$ be a constructible
sheaf on $\mathbb{P}^{1}$ with $l\mathcal{F}=0$. Suppose that $\mathcal{F}=g_{*}\mathcal{F}_{0}$
for some sheaf $\mathcal{F}_{0}$ on $\eta=\mathrm{Spec}F$, where
$g:\eta\longrightarrow S$. Then 
\[
h^{2}\left(\mathcal{F}\right)\le h^{1}\left(\mathcal{F}\right)-h^{0}\left(\mathcal{F}\right)+\mathrm{rk}_{l}\mathcal{F}_{\overline{\eta}}.
\]
\end{lem}
\begin{proof}
By \cite{milnedualitytheorems} Theorem 2.13 p. 174 we have the Euler
characteristic $\chi\left(\mathcal{F}\right)=h^{0}\left(\mathcal{F}\right)-h^{1}\left(\mathcal{F}\right)+h^{2}\left(\mathcal{F}\right)-h^{3}\left(\mathcal{F}\right)=0$.
Also from \cite{milnedualitytheorems} Corollary 3 p.177 we have for
any such sheaf $\mathcal{F}$ that $H^{3}\left(\mathcal{F}\right)\cong H^{0}\left(\mathcal{F}^{D}\right)$
where $\mathcal{F}^{D}$ is the dual sheaf defined by $\mathcal{F}^{D}\left(W\right)=\mathrm{Hom}\left(\mathcal{F}\mid_{W},\mu_{l}\mid_{W}\right)$.
So 
\begin{eqnarray*}
H^{0}\left(\mathcal{F}^{D}\right) & = & \mathrm{Hom}\left(\mathcal{F},\mu_{l}\right)\\
 & = & \mathrm{Hom}\left(j_{*}\mathcal{F}_{0},j_{*}\mu_{l}\right)\\
 & = & \mathrm{Hom}\left(j^{*}j_{*}\mathcal{F}_{0},\mu_{l}\right)\\
 & = & \mathrm{Hom}\left(\mathcal{F}_{0},\mu_{l}\right)\\
 & = & \mathrm{Hom}_{G_{F}}\left(\mathcal{F}_{\overline{\eta}},\mu_{l}\right)\\
 & \subset & \mathrm{Hom}_{\mathbb{Z}/l\mathbb{Z}}\left(\mathcal{F}_{\overline{\eta}},\mu_{l}\right).
\end{eqnarray*}
and the $l$-rank of this last set is bounded by $\mathrm{rk}_{l}\mathcal{F}_{\overline{\eta}}$.
Thus $h^{3}\left(\mathcal{F}\right)\le\mathrm{rk}_{l}\mathcal{F}_{\overline{\eta}}$
and the result follows. 
\end{proof}

\subsection{Number fields}

Now let $S=\mathrm{Spec}\mathbb{Z}\left[1/l\right]$ and $F=\mathbb{Q}$.
We have the following lemmas analogous to the function field case.

\begin{lem}
\label{lem:unit group mod number fields}Let $K$ be a number field
and let $s=r_{1}+r_{2}-1$ where $r_{1}$ and $r_{2}$ are the number
of real and complex embeddings. Let $u$ be the number of primes in
$K$ above $l$. Let $t=1$ if $l$ divides the order of $\mu_{K}$
and 0 otherwise. Then
\end{lem}
\[
\mathcal{O}_{K}^{\times}\left[1/l\right]/\left(\mathcal{O}_{K}^{\times}\left[1/l\right]\right)^{l}=\left(\mathbb{Z}/l\mathbb{Z}\right)^{s+u+t}.
\]

\begin{proof}
This follows from the assumption and since $\mathcal{O}_{K}^{\times}\left[1/l\right]\cong\mu_{K}\oplus\mathbb{Z}^{s+u}$.\end{proof}
\begin{lem}
\label{lem:mu_p cohomology number fields}For the sheaf $\mu_{l}$
on S we have

\[
H^{0}\left(\mu_{l}\right),H^{1}\left(\mu_{l}\right),H^{2}\left(\mu_{l}\right)=\begin{cases}
\mathbb{Z}/2\mathbb{Z},\left(\mathbb{Z}/2\mathbb{Z}\right)^{2},\mathbb{Z}/2\mathbb{Z} & \mbox{if }l=2\\
0,\mathbb{Z}/l\mathbb{Z},0 & \mbox{if }l\neq2
\end{cases}
\]
\end{lem}
\begin{proof}
Take cohomology of the Kummer sequence on Spec$\mathbb{Z}\left[1/l\right]$:

\[\begin{tikzcd}[row sep=1em]&\\ 
1\arrow{r} & H^{0}\left(\mu_{l}\right) \arrow{r} & \mathbb{Z}/2\mathbb{Z}\oplus\mathbb{Z} \arrow{r} & \mathbb{Z}/2\mathbb{Z}\oplus\mathbb{Z}\arrow{lld}\\ 
&H^{1}\left(\mu_{l}\right) \arrow{r} & 0 \arrow{r} & 0 \arrow{lld}\\
&H^{2}\left(\mu_{l}\right) \arrow{r} & \mathbb{Z}/2\mathbb{Z} \arrow{r} & \mathbb{Z}/2\mathbb{Z}\arrow{lld}\\
&H^{3}\left(\mu_{l}\right) \arrow{r} & 0 \arrow{r} & 0 \\ & \end{tikzcd}\]We have used that $H^{1}\left(\mbox{Spec}\mathbb{Z}\left[1/l\right],\mathbb{G}_{m}\right)=\mathrm{Pic}\left(\mbox{Spec}\mathbb{Z}\left[1/l\right]\right)=Cl_{\left\{ l\right\} }\left(\mathbb{Q}\right)=0$.
Also $H^{2}\left(\mbox{Spec}\mathbb{Z}\left[1/l\right],\mathbb{G}_{m}\right)=\mathbb{Z}/2\mathbb{Z}$
and $H^{3}\left(\mbox{Spec}\mathbb{Z}\left[1/l\right],\mathbb{G}_{m}\right)=0$
(see \cite{Milneetalecohomology}, p.109).\end{proof}
\begin{lem}
\label{lem:euler char number fields}Let $\mathcal{F}$ be a constructible
sheaf on $S=\mbox{Spec}\mathbb{Z}\left[1/l\right]$ with $l\mathcal{F}=0$.
Suppose that $\mathcal{F}=g_{*}\mathcal{F}_{0}$ for some sheaf $\mathcal{F}_{0}$
on $\eta=\mathrm{Spec}F$, where $g:\eta\longrightarrow S$. Furthermore
assume that $\mathcal{F}\left(\mathbb{R}\right)=0$ if $l\neq2$ and
$\mathcal{F}$ trivializes on $\mathbb{R}$ if $l=2$. Then 
\[
h^{2}\left(\mathcal{F}\right)\le h^{1}\left(\mathcal{F}\right)-h^{0}\left(\mathcal{F}\right).
\]
\end{lem}
\begin{proof}
By \cite{milnedualitytheorems} Theorem 2.13, p.174 we have the Euler
characteristic $\chi\left(\mathcal{F}\right)=h^{0}\left(\mathcal{F}\right)-h^{1}\left(\mathcal{F}\right)+h^{2}\left(\mathcal{F}\right)-h^{3}\left(\mathcal{F}\right)$
is equal to 
\begin{eqnarray*}
\chi\left(\mathcal{F}\right) & = & \mathrm{rk}_{l}\mathcal{F}\left(\mathbb{R}\right)-\mathrm{rk}_{l}H_{T}^{0}\left(G\left(\mathbb{C}/\mathbb{R}\right),\mathcal{F}\right)-\mathrm{rk}_{l}\mathcal{F}\left(L'\right).
\end{eqnarray*}
where $R$ is a finite set of points in $S$ such that $\mathcal{F}$
is locally constant on $S\backslash R$ and $L'$ is the maximal subfield
of $\overline{\mathbb{Q}}$ which is ramified only at $R$. So we
can choose $R$ such that $L\subset L'$ where $L$ is some minimal
number field on which $\mathcal{F}$ trivializes. Then $\mathcal{F}\left(L'\right)=\mathcal{F}_{\overline{\eta}}$.
Finally $H_{T}^{0}$ is the Tate cohomology group and if $l=2$ and
$\mathcal{F}$ trivializes on $\mathbb{R}$ we get $N_{G\left(\mathbb{C}/\mathbb{R}\right)}\mathcal{F}_{\mathrm{Spec}\mathbb{R}}^{G\left(\mathbb{C}/\mathbb{R}\right)}=0$,
so $H_{T}^{0}\left(G\left(\mathbb{C}/\mathbb{R}\right),\mathcal{F}\right)=0$.
Thus in both cases $\chi=-\mathrm{rk}_{l}\mathcal{F}_{\overline{\eta}}$.
So 
\[
h^{2}\left(\mathcal{F}\right)=h^{1}\left(\mathcal{F}\right)-h^{0}\left(\mathcal{F}\right)+h^{3}\left(\mathcal{F}\right)-\mathrm{rk}_{l}\mathcal{F}_{\overline{\eta}}.
\]
As in the proof of Lemma \ref{lem:euler char function field} $H^{3}\left(\mathcal{F}\right)\cong H^{0}\left(\mathcal{F}^{D}\right)$
and $\mathrm{rk}_{l}H^{0}\left(\mathcal{F}^{D}\right)\le\mathrm{rk}_{l}\mathcal{F}_{\overline{\eta}}$.
\end{proof}

\section{Exact Sequences of Sheaves\label{sec:Exact-Sequences-of}}

We now consider two schemes $\pi_{i}:X_{i}\longrightarrow S$ for
$i=1,2$ of the form desribed in Section \ref{sec:Preliminary-Results},
with generic points Spec$K_{i}$. Let $L/F$ be a Galois extension
containing the $K_{i}$ and let $G=\mathrm{Gal}\left(L/F\right)$.
We let $L'=L\left(\mu_{l}\right)$ and assume $K_{2}$ is disjoint
from $F\left(\mu_{l}\right)$. Let $B=\mathrm{Gal}\left(F\left(\mu_{l}\right)/F\right)$.
Then either $B=\left(\mathbb{Z}/l\mathbb{Z}\right)^{\times}$ or $B=1$.
Let $G'=\mbox{Gal}\left(L'/F\right)=G\times B$. Note $\left(\left|B\right|,l\right)=1$.

Let $\mathcal{L}_{i}=\left(\pi_{i}\right)_{*}\mu_{l}$ . Let $M=\mathcal{L}_{1,\overline{\eta}}$
and $N=\mathcal{L}_{2,\overline{\eta}}$. These are $G_{F}$-modules
where the action factors through $G\times B$ and $B$ acts by multiplication.
Since the $B$ action will have no bearing on the results in this
section we disregard it for now and simply talk about $G$-modules.
The $B$ action will become important in the next section when we
are computing cohomology.

In each case we apply Lemma \ref{lem:Let--be-2} to give a description
of $M$ and $N$ and then compute a series of exact sequences relating
them. 
\begin{rem}
\label{rem:takinginvariants}Since all modules will be $G$-modules,
by Lemma \ref{lem:Let--be-3} they will represent sheaves on $U$,
the set of points where neither $\pi_{1}$ or $\pi_{2}$ ramify. We
will then apply the functor $j_{*}$ to obtain sheaves on $S$ (noting
by \ref{rem:By-the-above} that this recovers the sheaves $\mathcal{L}_{i}$).
To check exactness of a sequence of sheaves is equivalent to checking
exactness of stalks at each geometric point. Thus by Lemma \ref{lem:j_*stalks}
we need to check the original sequences of modules remain exact after
taking $I_{z}$-invariants by the inertia group $I_{z}$ of each ramified
point $z$.
\end{rem}

\subsection{The case of $S_{3}$}

Let $l=3$. Let $L/F$ be a Galois extension with $\mathrm{Gal}\left(L/F\right)=S_{3}$.
Let $K_{1}$ be a non-Galois cubic subfields. Let $K_{2}$ be the
unique quadratic subfield (this is called the quadratic resolvent
of $K_{1}$).

Let $S_{3}=\left\langle \sigma,\tau\mid\sigma^{2}=\tau^{3}=1,\sigma\tau\sigma=\tau^{-1}\right\rangle $
be the usual presentation. Then as $S_{3}$-modules 
\begin{eqnarray*}
M & = & \mathbb{Z}/3\mathbb{Z}\left\langle 1,\tau,\tau^{2}\right\rangle 
\end{eqnarray*}
 where $\tau^{i}$ represents the coset modulo the subgroup $\left\langle \sigma\right\rangle \subset S_{3}$
and similarly 
\begin{eqnarray*}
N & = & \mathbb{Z}/3\mathbb{Z}\left\langle 1,\sigma\right\rangle 
\end{eqnarray*}
 where $\sigma^{i}$ represents the coset modulo the subgroup $\left\langle \tau\right\rangle \subset S_{3}$. 

We now write down a series of exact sequences relating these modules.
Let $N'=\mathbb{Z}/3\mathbb{Z}\left\langle 1-\sigma\right\rangle $.
Let $T$ be the module representing the sheaf $\mu_{3}$. It is one
dimensional with trivial $G$ action. Then we have 

\begin{equation}
\begin{gathered}
\begin{tikzcd}[row sep=0.5em]&\\ 
0\arrow{r} & \left\langle 1+\tau+\tau^{2}\right\rangle  \arrow{r}& M \arrow{r} & M' \arrow{r} & 0\\ 
0\arrow{r}& N' \arrow{r} & M' \arrow{r} & T \arrow{r}& 0. \\&
\end{tikzcd}
\end{gathered}
\end{equation}\label{s3sequences}The map $N'\longrightarrow M'$ is given by $1-\sigma\longmapsto1-\tau$.
We also have 
\begin{equation}
N=N'\oplus\left\langle 1+\sigma\right\rangle .\label{s3directsum}
\end{equation}
 Note that $T$ corresponds to the sheaf $\mu_{3}$.

This is a sequence of sheaves at the generic point. Now we apply \ref{rem:takinginvariants}.
Here the assumption that $L/K_{2}$ is unramified implies that the
inertia group at any point has order $2$. It is a fact that $H^{1}\left(G,M\right)=0$
whenever the order of $G$ is coprime to $M$. Since all of the above
modules have order a power of $3$ we obtain exact sequences of sheaves
on $S$. Note also that $T$ will be extended to $\mu_{3}$ on $S$.

\subsection{The case of $S_{4}$ and $A_{4}$.}

Let $L/F$ be a Galois extensions with $\mathrm{Gal}\left(L/F\right)=S_{4}$.
Let $K_{1}$ be a non-Galois quartic subfield. There is a unique subgroup
of $S_{4}$ isomorphic to the Klein-four group $V_{4}$ which is furthermore
normal in $S_{4}$. Its fixed field has Galois group $S_{3}$, hence
has 3 cubic subfields. Let $K_{2}$ be one of these fields (this is
called the cubic resolvent of $K_{1}$).

Then 
\[
M=\mathbb{Z}/2\mathbb{Z}\left\langle a_{1},a_{2},a_{3},a_{4}\right\rangle ,
\]
\[
N=\mathbb{Z}/2\mathbb{Z}\left\langle b_{1},b_{2},b_{3}\right\rangle .
\]
The action of $S_{4}$ permutes the $a_{i}$'s and factors through
$S_{3}$ on the $b_{i}$'s. 

Let $M_{1}=\left\langle a_{2}+a_{1},a_{3}+a_{1},a_{4}+a_{1}\right\rangle $
and $N'=\mathbb{Z}/2\mathbb{Z}\left\langle a_{1}+a_{2},a_{1}+a_{3}\right\rangle /\left\langle a_{1}+a_{2}+a_{3}+a_{4}\right\rangle $.
Let $T$ be the trivial module of dimension 1. Then we have

\begin{equation}
\begin{gathered}
\begin{tikzcd}[row sep=0.5em]
0\arrow{r}& M_{1} \arrow{r} & M \arrow{r} & T \arrow{r}& 0 \\
0\arrow{r} & T  \arrow{r}& M_{1} \arrow{r} & N' \arrow{r} & 0
\end{tikzcd}
\end{gathered}
\end{equation}\label{S4sequence}

Furthermore we have 
\begin{equation}
N=N'\oplus T\label{eq:S4directsum}
\end{equation}
where $N'$ lies in $N$ as $a_{1}+a_{2}\longmapsto b_{3}$ and $a_{1}+a_{3}\longmapsto b_{2}$.

Now we apply Remark \ref{rem:takinginvariants}. We assume that the
inertia group $I$ of any point is disjoint from $V_{4}$. Then the
sequences of $I$-invariants remain exact. This is clear for any subgroup
of order 3. It is straightforward to check that for any subgroup $I$
generated by a transposition of $S_{4}$, $\mbox{dim}_{\mathbb{F}_{2}}M^{I}=3$.
Then $\mbox{dim}_{\mathbb{F}_{2}}M_{1}{}^{I}\ge2$ and since $M{}_{1}^{I}\ne M_{1}$
it must be that $\mbox{dim}_{\mathbb{F}_{2}}M_{1}{}^{I}=2$ so the
top sequence remains exact. It is also easy to check that $\mbox{dim}_{\mathbb{F}_{2}}N'^{I}=1$
so the middle sequence remains exact. This also implies that exactness
is preserved under $I=S_{3}\subset S_{4}$ since then $I$ is generated
by a transposition and a 3-cycle. 

In the case of $\mathrm{Gal}\left(L/F\right)=A_{4}$ there is also
a unique normal subgroup isomorphic to $V_{4}$. Let $K_{2}$ be its
fixed field which is Galois with group $C_{3}$. Again $K_{1}$ is
one of the quartic subfields. We also assume that $I$ is disjoint
from $V_{4}$. Then the sequences relating $M$ and $N$ are exactly
the same as above and the argument for the $I$-invariants is similar.

\subsection{The case of $D_{2l}$}

We can generalize the above method to the following case. Let $l$
be an odd prime. Let $D_{2l}=\left\langle \sigma,\tau\mid\sigma^{2}=\tau^{l}=1,\sigma\tau\sigma=\tau^{-1}\right\rangle $.
Let $L/F$ be a galois extension with $\mathrm{Gal}\left(L/F\right)=D_{2l}$.
Let $K_{1}$ be a non-Galois subextension of degree $l$. Let $K_{2}$
be the unique quadratic subfield of $L$.

Then 
\[
M=\mathbb{Z}/l\mathbb{Z}\left\langle 1,\tau,\ldots,\tau^{l-1}\right\rangle 
\]

\[
N=\mathbb{Z}/l\mathbb{Z}\left\langle 1,\sigma\right\rangle 
\]
where the basis elements are assumed to represent the cosets of $D_{2l}$
by the appropriate subgroup.

We want to write down a series of exact sequences relating these two
modules. We start by computing the Jordan-Holder decomposition of
$M$. For each $k=0,\ldots,l-1$ let 
\[
M_{k}=\left\{ \sum_{i=0}^{l-1}a\left(i\right)\tau^{i}\mid a\left(x\right)\mbox{ a degree \ensuremath{k} polynomial}\right\} 
\]
where the coefficients are taken in $\mathbb{Z}/l\mathbb{Z}$. Then
it is not hard to see that this is a filtration of $G$-modules, with
\[
M_{k}/M_{k-1}=\begin{cases}
\mathbb{Z}/l\mathbb{Z} & k\mbox{ even}\\
N' & k\mbox{ odd}
\end{cases}
\]
 where $N'=\mathbb{Z}/l\mathbb{Z}\left\langle 1-\sigma\right\rangle $
as in the $S_{3}$ example. Furthermore $M_{0}=\mathbb{Z}/l\mathbb{Z}$
and $M_{l-1}=M$. This can be restated as the series of exact sequences

\begin{equation}
\begin{gathered}
\begin{tikzcd}[row sep=0.5em]
0\arrow{r}& M_{l-2} \arrow{r} & M_{l-1} \arrow{r} & \mathbb{Z}/l\mathbb{Z} \arrow{r}& 0 \\
0\arrow{r} & M_{l-3}  \arrow{r}& M_{l-2} \arrow{r} & N' \arrow{r} & 0\\
&&\vdots\\
0\arrow{r}& M_{0} \arrow{r} & M_{1} \arrow{r} & N' \arrow{r}& 0 \\
0\arrow{r}& 0  \arrow{r} & M_{0} \arrow{r} & \mathbb{Z}/l\mathbb{Z} \arrow{r}&0 
\end{tikzcd}
\end{gathered}
\end{equation}\label{D2psequence1}

We also have 
\begin{equation}
N=N'\oplus\mathbb{Z}/l\mathbb{Z}.\label{D2psequence2}
\end{equation}

Now we apply Remark \ref{rem:takinginvariants}. Again using the fact
that $H^{1}\left(G,M\right)=0$ whenever the order of $G$ is coprime
to $M$ for any $G$-module $M$, we see that taking invariants by
any subgroup of $G$ of order coprime to $l$ preserves exactness.
\\

\subsection{The case of $\mathbb{Z}/l\mathbb{Z}\rtimes\mathbb{Z}/r\mathbb{Z}$. }

Let $L/F$ be a Galois extensions with $G=\mathbb{Z}/l\mathbb{Z}\rtimes\mathbb{Z}/r\mathbb{Z}$.
We use the presentation of $G=\left\{ \sigma,\tau\mid\sigma^{r}=\tau^{l}=1,\sigma\tau\sigma^{-1}=\tau^{\zeta}\right\} $
where $\zeta\in\left(\mathbb{Z}/l\mathbb{Z}\right)^{\times}$ denotes
an $r$th root of unity (since $r\mid l-1$). We assume $l\equiv1$
mod $r$. Let $K_{1}$ be one of the conjugate degree $l$ subfields
and let $K_{2}$ be the degree $r$ Galois subfield. 

Then $M=\mathbb{Z}/l\mathbb{Z}\left\langle 1,\tau,\ldots,\tau^{l-1}\right\rangle $
and $N=\mathbb{Z}/l\mathbb{Z}\left\langle 1,\sigma,\ldots,\sigma^{r-1}\right\rangle $. 

We can describe a filtration of $N$ as follows. For each $k=0,\ldots,r-1$
let 
\[
N_{k}=\left\{ \sum_{i=0}^{r-1}a\left(\zeta^{i}\right)\sigma^{i}\mid a\left(x\right)\mbox{ a degree \ensuremath{k} polynomial}\right\} 
\]
where the coefficients are taken in $\mathbb{Z}/l\mathbb{Z}$. Then
it is not hard to see that this is a filtration of $G$-modules (with
$\tau$ acting trivially), with 
\[
N_{k}/N_{k-1}=R_{k}
\]
 where $R_{k}=\mathbb{Z}/l\mathbb{Z}$ on which $\sigma$ acts as
multiplication by $\zeta^{-k}$. Furthermore $N_{0}=\mathbb{Z}/l\mathbb{Z}$
and $N_{r-1}=N$. Since $\left\langle \sigma\right\rangle $ is cyclic
and has order coprime to $l$, the representation $N$ decomposes
as a direct sum of 1 dimensional representations. This implies that
\begin{equation}
N=\bigoplus_{k=0}^{r-1}R_{k}.\label{semidirect directsum}
\end{equation}

Similarly we describe a filtration of $M$. For each $k=0,\ldots,l-1$
let 
\[
M_{k}=\left\{ \sum_{i=0}^{l-1}a\left(i\right)\tau^{i}\mid a\left(x\right)\mbox{ a degree \ensuremath{k} polynomial}\right\} .
\]
This is a filtration of $G$-modules with the same factors 
\[
M_{k}/M_{k-1}=R_{k\left(\mathrm{mod}l\right)}.
\]
 Note however that here each $R_{k}$ appears $n=\left(l-1\right)/r$
times except $R_{0}$ which appears $n+1$ times. Here $M_{l-1}=M$.
This can be restated as the series of exact sequences

\begin{equation}
\begin{gathered}
\begin{tikzcd}[row sep=0.5em]
0\arrow{r}& M_{l-2} \arrow{r} & M_{l-1} \arrow{r} & \mathbb{Z}/l\mathbb{Z} \arrow{r}& 0 \\
0\arrow{r} & M_{l-3}  \arrow{r}& M_{l-2} \arrow{r} & R_{l-1} \arrow{r} & 0\\
&&\vdots\\
0\arrow{r}& M_{0} \arrow{r} & M_{1} \arrow{r} & R_{1} \arrow{r}& 0 \\
0\arrow{r}& 0  \arrow{r} & M_{0} \arrow{r} & \mathbb{Z}/l\mathbb{Z} \arrow{r}&0 
\end{tikzcd}
\end{gathered}
\end{equation}\label{semidirectsequence2}

Now we apply Remark \ref{rem:takinginvariants}. If we assume that
$L$ has inertia group coprime to $l$ for every prime ramified in
$L/\mathbb{Q}$ then the sequences of inertial invariants remains
exact.

\section{Taking Cohomology in Function Fields}

In this section we take cohomology of the sequences computed in Section
\ref{sec:Exact-Sequences-of} to obtain relations between $H^{1}\left(\mathcal{L}_{1}\right)$
and $H^{2}\left(\mathcal{L}_{2}\right)$ which by Lemmas \ref{lem:pic and h1}
and \ref{lem:unit group mod p functionfield} give a relation between
$\mathrm{Pic}C_{1}$ and $\mathrm{Pic}C_{2}$.

\subsection{The case of $S_{3}$}
\begin{lem}
\label{lem:S3functionfield}If $G=S_{3}$ and $\mu_{3}\in\mathbb{F}_{p}$
then 
\[
\mathrm{rk}_{3}\mathrm{Pic}C_{2}-2\le\mathrm{rk}_{3}\mathrm{Pic}C_{1}\le\mathrm{rk}_{3}\mathrm{Pic}C_{2}.
\]
\end{lem}
\begin{proof}
By Lemma \ref{lem:mu_p cohomology} we have $H^{i}\left(\mu_{3}\right)=\mathbb{Z}/3\mathbb{Z}$
for $i=0,1,2$. Furthermore $H^{0}\left(M\right)=H^{0}\left(N\right)=\mathbb{Z}/3\mathbb{Z}$
and $H^{0}\left(M'\right)=H^{0}\left(N'\right)=0$. 

Now we take cohomology of the sequences \ref{s3sequences}. Using
the same notation to denote the sheaves extended to $\mathbb{P}^{1}$
the first sequence gives

\[
\begin{tikzcd}[row sep=1em]
0\arrow{r} & \mathbb{Z}/3\mathbb{Z} \arrow{r} & \mathbb{Z}/3\mathbb{Z} \arrow{r} & 0 \arrow{lld}\\ 
& \mathbb{Z}/3\mathbb{Z} \arrow{r} & H^{1}\left(M\right) \arrow{r} & H^{1}\left(M'\right) \arrow{lld}\\
& \mathbb{Z}/3\mathbb{Z} \arrow{r} & H^{2}\left(M\right) \arrow{r} & H^{2}\left(M'\right) \\ & \end{tikzcd}\]and the second sequence gives 

\[
\begin{tikzcd}[row sep=1em]
0\arrow{r} & 0 \arrow{r} & 0 \arrow{r} & \mathbb{Z}/3\mathbb{Z} \arrow{lld}\\ 
& H^{1}\left(N'\right) \arrow{r} & H^{1}\left(M'\right) \arrow{r} & \mathbb{Z}/3\mathbb{Z} \arrow{lld}\\
& H^{2}\left(N'\right) \arrow{r} & H^{2}\left(M'\right) \arrow{r} & \mathbb{Z}/3\mathbb{Z}. \\ & \end{tikzcd}\]Thus we get the equalities 
\[
1+h^{1}\left(M'\right)=s+h^{1}\left(M\right),
\]
\[
1+h^{1}\left(M'\right)=t+h^{1}\left(N'\right)
\]
where $s\le1$ and $t\le1$. From \ref{s3directsum} we also get 
\[
1+h^{1}\left(N'\right)=h^{1}\left(N\right).
\]
 Putting these together gives the bounds
\[
h^{1}\left(N\right)-2\le h^{1}\left(M\right)\le h^{1}\left(N\right).
\]
 Thus using Lemma \ref{lem:pic and h1} we have $\mathrm{rk}_{3}\mathrm{Pic}\left(C_{1}\right)=h^{1}\left(M\right)-1$
from the begining, and similarly for $C_{2}$.
\end{proof}
Now suppose $\mu_{3}\notin\mathbb{F}_{p}$. Then by Lemma \ref{lem:mu_p cohomology}
we have $H^{i}\left(\mu_{3}\right)=0$ for $i=0,1$ and $H^{2}\left(\mu_{3}\right)=\mathbb{Z}/3\mathbb{Z}$.
Furthermore $H^{0}\left(M\right)=H^{0}\left(N\right)=H^{0}\left(M'\right)=H^{0}\left(N'\right)=0$
since these are the $G$-invariants of the stalk and in this case
$G$ has an element which acts as multiplication by 2, which fixes
no element of any of the stalks. A similar computation to the above
gives:
\begin{lem}
If $G=S_{3}$ and $\mu_{3}\notin\mathbb{F}_{p}$ then 
\[
\mathrm{rk}_{3}\mathrm{Pic}C_{2}-1\le\mathrm{rk}_{3}\mathrm{Pic}C_{1}\le\mathrm{rk}_{3}\mathrm{Pic}C_{2}.
\]

\end{lem}

\subsection{The case of $S_{4}$ and $A_{4}$}
\begin{lem}
If $G=S_{4}$ or $A_{4}$ then 
\end{lem}
\[
\mathrm{rk}_{2}\mathrm{Pic}C_{2}-2\le\mathrm{rk}_{2}\mathrm{Pic}C_{1}\le\mathrm{rk}_{2}\mathrm{Pic}C_{2}.
\]

\begin{proof}
By Lemma \ref{lem:mu_p cohomology} we have $H^{i}\left(\mu_{2}\right)=\mathbb{Z}/2\mathbb{Z}$
for $i=0,1,2$. Furthermore $H^{0}\left(M\right)=H^{0}\left(N\right)=\mathbb{Z}/2\mathbb{Z}$.
We consider the filtration \ref{S4sequence}. Taking cohomology of
the first sequences gives

\[
\begin{tikzcd}[row sep=1em]
0\arrow{r} & H^{0}\left(M_{1}\right) \arrow{r} & \mathbb{Z}/2\mathbb{Z} \arrow{r} & \mathbb{Z}/2\mathbb{Z} \arrow{lld}\\ 
& H^{1}\left(M_{1}\right) \arrow{r} & H^{1}\left(M\right) \arrow{r} & \mathbb{Z}/2\mathbb{Z} \arrow{lld}\\
& H^{2}\left(M_{1}\right) \arrow{r} & H^{2}\left(M\right) \arrow{r} & \mathbb{Z}/2\mathbb{Z} \\ & \end{tikzcd}\]

and the second sequence gives

\[
\begin{tikzcd}[row sep=1em]
0\arrow{r} & \mathbb{Z}/2\mathbb{Z} \arrow{r} & H^{0}\left(M_{1}\right) \arrow{r} & 0 \arrow{lld}\\ 
& \mathbb{Z}/2\mathbb{Z} \arrow{r} & H^{1}\left(M_{1}\right) \arrow{r} & H^{1}\left(N'\right) \arrow{lld}\\
& \mathbb{Z}/2\mathbb{Z} \arrow{r} & H^{2}\left(M_{1}\right) \arrow{r} & H^{2}\left(N'\right) \\ & \end{tikzcd}\]These sequences give the equalities 
\[
h^{1}\left(M\right)+1=h^{1}\left(M_{1}\right)+s
\]
\[
h^{1}\left(M_{1}\right)+t=h^{1}\left(N'\right)+1
\]
where $s\le1$, $t\le1$. From \ref{eq:S4directsum} we also get 
\[
1+h^{1}\left(N'\right)=h^{1}\left(N\right).
\]
Combining these gives the bounds 
\[
h^{1}\left(N\right)-2\le h^{1}\left(M\right)\le h^{1}\left(N\right).
\]

\end{proof}

\subsection{The case of $D_{2l}$}
\begin{lem}
If $G=D_{2l}$ and $\mu_{l}\in\mathbb{F}_{p}$ then 
\[
\mathrm{rk}_{l}\mathrm{Pic}\left(C_{2}\right)-\frac{l-1}{2}\le\mathrm{rk}_{l}\mathrm{Pic}\left(C_{1}\right)\le\frac{l-1}{2}\mathrm{rk}_{l}\mathrm{Pic}\left(C_{2}\right)+1.
\]
\end{lem}
\begin{proof}
Since $\mu_{l}\in\mathbb{F}_{p}$ in this case $B$ is trivial, $L=L'$
and we are just looking at sequences of $D_{2l}$-modules. Then by
Lemma \ref{lem:mu_p cohomology} $H^{i}\left(\mu_{l}\right)=\mathbb{Z}/l\mathbb{Z}$
for $i=0,1,2$. Furthermore $H^{0}\left(N\right)=H^{0}\left(M\right)=\mathbb{Z}/l\mathbb{Z}$.
This implies that $H^{0}\left(M_{k}\right)=\mathbb{Z}/l\mathbb{Z}$
for all $k$ and $H^{0}\left(N'\right)=0$.

We consider the filtration \ref{D2psequence1}. Taking cohomology
of the first sequence gives

\[
\begin{tikzcd}[row sep=1em]
0\arrow{r} & \mathbb{Z}/l\mathbb{Z} \arrow{r} & \mathbb{Z}/l\mathbb{Z} \arrow{r} & \mathbb{Z}/l\mathbb{Z} \arrow{lld}\\ 
& H^{1}\left(M_{l-2}\right) \arrow{r} & H^{1}\left(M_{l-1}\right) \arrow{r} & \mathbb{Z}/l\mathbb{Z} \arrow{lld}\\
& H^{2}\left(M_{l-2}\right) \arrow{r} & H^{2}\left(M_{l-1}\right) \arrow{r} & \mathbb{Z}/l\mathbb{Z}. \\ & \end{tikzcd}\]Taking cohomology of the second sequence gives 

\[
\begin{tikzcd}[row sep=1em]
0\arrow{r} & \mathbb{Z}/l\mathbb{Z} \arrow{r} & \mathbb{Z}/l\mathbb{Z} \arrow{r} & 0 \arrow{lld}\\ 
& H^{1}\left(M_{l-3}\right) \arrow{r} & H^{1}\left(M_{l-2}\right) \arrow{r} & H^{1}\left(N'\right) \arrow{lld}\\
& H^{2}\left(M_{l-3}\right) \arrow{r} & H^{2}\left(M_{l-2}\right) \arrow{r} & H^{2}\left(N'\right). \\ & \end{tikzcd}\]And so on. The $k=1$ sequence gives 

\[
\begin{tikzcd}[row sep=1em]
0\arrow{r} & \mathbb{Z}/l\mathbb{Z} \arrow{r} & \mathbb{Z}/l\mathbb{Z} \arrow{r} & 0 \arrow{lld}\\ 
& \mathbb{Z}/l\mathbb{Z} \arrow{r} & H^{1}\left(M_{1}\right) \arrow{r} & H^{1}\left(N'\right) \arrow{lld}\\
& \mathbb{Z}/l\mathbb{Z} \arrow{r} & H^{2}\left(M_{1}\right) \arrow{r} & H^{2}\left(N'\right). \\ & \end{tikzcd}\] Hence we get the equalities

\[
h^{1}\left(M_{k}\right)+1=h^{1}\left(M_{k-1}\right)+s_{k}\mbox{ if \ensuremath{k} is even},
\]
\[
h^{1}\left(M_{k}\right)=t_{k}+h^{1}\left(M_{k-1}\right)\mbox{ if \ensuremath{k} is odd},
\]
\[
v+h^{1}\left(M_{1}\right)=h^{1}\left(N'\right)+1
\]

where $s_{k}\le1$, $t_{k}\le h^{1}\left(N'\right)$ and $v\le1$.
By \ref{D2psequence2} we also have
\[
h^{1}\left(N\right)=h^{1}\left(N'\right)+1.
\]
Putting all of these together gives the bounds 
\[
h^{1}\left(N\right)-\frac{l-1}{2}-1\le h^{1}\left(M\right)\le\frac{l-1}{2}\left(h^{1}\left(N\right)-1\right)+1.
\]

\end{proof}
Now suppose $\mu_{l}\notin\mathbb{F}_{p}$. In this case $B$ is non-trivial.
Then by Lemma \ref{lem:mu_p cohomology} $H^{i}\left(\mu_{l}\right)=0$
for $i=0,1$ and $H^{2}\left(\mu_{l}\right)=\mathbb{Z}/l\mathbb{Z}$.
Furthermore $H^{0}\left(M\right)=H^{0}\left(N\right)=0$. This implies
that $H^{0}\left(M_{k}\right)=0$ for all $k$ and $H^{0}\left(N'\right)=0$.
A similar computation gives:
\begin{lem}
If $G=D_{2l}$ and $\mu_{l}\notin\mathbb{F}_{p}$ then
\[
\mathrm{rk}_{l}\mathrm{Pic}\left(C_{2}\right)-1\le\mathrm{rk}_{l}\mathrm{Pic}\left(C_{1}\right)\le\frac{l-1}{2}\mathrm{rk}_{l}\mathrm{Pic}\left(C_{2}\right).
\]

\end{lem}

\subsection{The case of $\mathbb{Z}/l\mathbb{Z}\rtimes\mathbb{Z}/r\mathbb{Z}$}
\begin{lem}
If $G=\mathbb{Z}/l\mathbb{Z}\rtimes\mathbb{Z}/r\mathbb{Z}$ and $\mu_{l}\in\mathbb{F}_{p}$
then 
\[
\frac{1}{r-1}\left(\mathrm{rk}_{l}\mathrm{Pic}C_{2}\right)-\frac{l-3}{r-1}-2\le\mathrm{rk}_{l}\mathrm{Pic}C_{1}\le\frac{\left(l-1\right)}{r}\left(\mathrm{rk}_{l}\mathrm{Pic}C_{2}+1\right).
\]
\end{lem}
\begin{proof}
Suppose $\mu_{l}\in\mathbb{F}_{p}$. Then by Lemma \ref{lem:mu_p cohomology}
we have $H^{i}\left(\mu_{l}\right)=\mathbb{Z}/l\mathbb{Z}$ for $i=0,1,2$
. Furthermore $H^{0}\left(M\right)=H^{0}\left(N\right)=\mathbb{Z}/l\mathbb{Z}$.
This implies that $H^{0}\left(M_{k}\right)=\mathbb{Z}/l\mathbb{Z}$
for all $k$ and $H^{0}\left(N_{k}\right)=\mathbb{Z}/l\mathbb{Z}$
for all $k$. 

By \ref{semidirect directsum} and since cohomology commutes with
direct sums we get 
\[
h^{1}\left(N\right)=\sum_{k=0}^{r-1}h^{1}\left(R_{k}\right).
\]

Now we focus on the filtration \ref{semidirectsequence2}.  Taking
cohomology of the $k$th sequence when $k\neq1$ and $R_{k}\neq\mathbb{Z}/l\mathbb{Z}$
gives:

\[
\begin{tikzcd}[row sep=1em]
0\arrow{r} & \mathbb{Z}/l\mathbb{Z} \arrow{r} & \mathbb{Z}/l\mathbb{Z} \arrow{r} & 0 \arrow{lld}\\ 
& H^{1}\left(M_{k-1}\right) \arrow{r} & H^{1}\left(M_{k}\right) \arrow{r} & H^{1}\left(R_{k}\right) \arrow{lld}\\
& H^{2}\left(M_{k-1}\right) \arrow{r} & H^{2}\left(M_{k}\right) \arrow{r} & H^{2}\left(R_{k}\right) \\ & \end{tikzcd}\]and when $k=1$ gives:

\[
\begin{tikzcd}[row sep=1em]
0\arrow{r} & \mathbb{Z}/l\mathbb{Z} \arrow{r} & \mathbb{Z}/l\mathbb{Z} \arrow{r} & 0 \arrow{lld}\\ 
& \mathbb{Z}/l\mathbb{Z} \arrow{r} & H^{1}\left(M_{1}\right) \arrow{r} & H^{1}\left(R_{1}\right) \arrow{lld}\\
& \mathbb{Z}/l\mathbb{Z} \arrow{r} & H^{2}\left(M_{1}\right) \arrow{r} & H^{2}\left(R_{1}\right) \\ & \end{tikzcd}\]and when $R_{k}=\mathbb{Z}/l\mathbb{Z}$ gives: 

\[
\begin{tikzcd}[row sep=1em]
0\arrow{r} & \mathbb{Z}/l\mathbb{Z} \arrow{r} & \mathbb{Z}/l\mathbb{Z} \arrow{r} & \mathbb{Z}/l\mathbb{Z} \arrow{lld}\\ 
& H^{1}\left(M_{k-1}\right) \arrow{r} & H^{1}\left(M_{k}\right) \arrow{r} & \mathbb{Z}/l\mathbb{Z} \arrow{lld}\\
& H^{2}\left(M_{k-1}\right) \arrow{r} & H^{2}\left(M_{k}\right) \arrow{r} & \mathbb{Z}/l\mathbb{Z}. \\ & \end{tikzcd}\]Now by Lemma \ref{lem:euler char function field} 
\begin{eqnarray*}
h^{2}\left(M_{k}\right) & \le & h^{1}\left(M_{k}\right)-h^{0}\left(M_{k}\right)+\mathrm{rk}_{l}M_{k}\\
 & = & h^{1}\left(M_{k}\right)+k
\end{eqnarray*}
 for all $k$. Hence for all $R_{k}\neq\mathbb{Z}/l\mathbb{Z}$ and
$k\ne1$

\[
s_{k}+h^{1}\left(M_{k}\right)=h^{1}\left(R_{k}\right)+h^{1}\left(M_{k-1}\right)
\]
 where $s_{k}\le h^{1}\left(M_{k-1}\right)+\left(k-1\right)$. And
for $k=1$

\[
s_{1}+h^{1}\left(M_{1}\right)=h^{1}\left(R_{1}\right)+1
\]
 where $s_{1}\le1$. And for all $R_{k}=\mathbb{Z}/l\mathbb{Z}$ 
\[
h^{1}\left(M_{k}\right)=t_{k}+h^{1}\left(M_{k-1}\right)-1
\]
where $t_{k}\le1$. Combining this gives the upper bound 
\begin{eqnarray*}
h^{1}\left(M\right) & \le & \frac{l-1}{r}\sum_{k=1}^{l-1}h^{1}\left(R_{k}\right)+1.
\end{eqnarray*}
 To get a lower bound we note that for $R_{k}\neq\mathbb{Z}/l\mathbb{Z}$
and $k\neq1$ we have both 
\begin{eqnarray*}
h^{1}\left(R_{k}\right)-\left(k-1\right) & \le & h^{1}\left(M_{k}\right),\\
h^{1}\left(M_{k-1}\right) & \le & h^{1}\left(M_{k}\right).
\end{eqnarray*}
and for $R_{k}=\mathbb{Z}/l\mathbb{Z}$ we have $h^{1}\left(M_{k-1}\right)-1\le h^{1}\left(M_{k}\right)$
so
\[
h^{1}\left(M\right)\ge\max_{k}\left\{ h^{1}\left(R_{k}\right)-\left(k-1\right)\right\} -1\ge\frac{1}{r-1}\sum_{k=1}^{r-1}h^{1}\left(R_{k}\right)-\frac{l-3}{r-1}-1
\]
(note in the above max $k$ is bounded above by $l-2$).  Putting
it all together we get
\[
\frac{1}{r-1}\left(h^{1}\left(N\right)-1\right)-\frac{l-3}{r-1}-1\le h^{1}\left(M\right)\le\frac{\left(l-1\right)}{r}\left(h^{1}\left(N\right)-1\right)+1.
\]

\end{proof}
Now suppose $\mu_{l}\notin\mathbb{F}_{p}$. Then by Lemma \ref{lem:mu_p cohomology}
we have $H^{i}\left(\mu_{l}\right)=0$ for $i=0,1$ and $H^{2}\left(\mu_{l}\right)=\mathbb{Z}/l\mathbb{Z}$
. Furthermore $H^{0}\left(M\right)=H^{0}\left(N\right)=0$. This implies
that $H^{0}\left(M_{k}\right)=0$ for all $k$ and $H^{0}\left(N_{k}\right)=0$
for all $k$. A similar computation gives:
\begin{lem}
If $G=\mathbb{Z}/l\mathbb{Z}\rtimes\mathbb{Z}/r\mathbb{Z}$ and $\mu_{l}\notin\mathbb{F}_{p}$
then 
\[
\frac{1}{r-1}\mathrm{rk}_{l}\mathrm{Pic}C_{2}-\frac{l-2}{r-1}.\le\mathrm{rk}_{l}\mathrm{Pic}C_{1}\le\frac{\left(l-1\right)}{r}\mathrm{rk}_{l}\mathrm{Pic}C_{2}.
\]

\end{lem}

\section{Taking Cohomology in Number Fields}

The same approach works in the case of number fields, with the curves
replaced by Spec$\mathcal{O}_{K}\left[1/l\right]$ where $\mathcal{O}_{K}$
is the ring of integers of the number field $K$. We adjoin the element
$1/l$ to ensure that the Kummer sequence for the prime $l$ is exact.
\begin{rem}
\label{sclassgroup}Let $S$ be a finite set of primes in $K$ and
let $Cl_{S}\left(K\right)$ be the class group of $K$ away from the
primes in $S$. The relation with the usual class group is as follows.
There is a surjective morphism 
\[
\phi:Cl\left(K\right)\longrightarrow Cl_{S}\left(K\right)
\]
 which sends primes in $S$ to the trivial class. Thus $\mbox{rk}_{l}Cl_{S}\left(K\right)+\mbox{rk}_{l}\ker\phi=\mbox{rk}_{l}Cl\left(K\right)$.
An element of $Cl\left(K\right)$ is in the kernel if it has a representative
supported only on primes in $S$. Thus $\mbox{rk}_{l}\ker\phi\le l^{\left|S\right|-1}$,
with the -1 in the exponent coming from the fact that there is always
at least one relation between the full set of primes in $K$ lying
above any prime.
\end{rem}
In this section we will let $S$ consist of the set of primes above
$l$. The above remark is imporant because $\mathrm{Pic}\left(\mathrm{Spec}\mathcal{O}_{K}\left[1/l\right]\right)=Cl_{S}\left(K\right)$.
We let $s_{i}=r_{1}\left(K_{i}\right)+r_{2}\left(K_{i}\right)-1$
and $u_{i}$ denote the number of primes in $K_{i}$ above $l$. Finally
let $t_{i}=s_{i}+u_{i}$. 

The computations are very similar to the previous section, and we
simply state the exact sequences and the final result in each case,
omitting computations.

\subsection{The case of $S_{3}$ }
\begin{lem}
If $G=S_{3}$ and $K_{2}\neq\mathbb{Q}\left(\zeta_{3}\right)$ then
\[
\mathrm{rk}_{3}Cl_{S}\left(K_{2}\right)-1+\left(t_{2}-t_{1}\right)\le\mathrm{rk}_{3}Cl_{S}\left(K_{1}\right)
\]
 and 
\[
\mathrm{rk}_{3}Cl_{S}\left(K_{1}\right)\le\mathrm{rk}_{3}Cl_{S}\left(K_{2}\right)+\left(t_{2}-t_{1}\right).
\]
\end{lem}
\begin{proof}
By Lemma \ref{lem:mu_p cohomology number fields} we have $H^{1}\left(\mu_{3}\right)=\mathbb{Z}/3\mathbb{Z}$
and $H^{i}\left(\mu_{3}\right)=0$ for $i=0,2$. Furthermore $H^{0}\left(M\right)=H^{0}\left(N\right)=0$
and $H^{0}\left(M'\right)=H^{0}\left(N'\right)$. We have:

We consider the filtration \ref{s3sequences}. Taking cohomology
of the first sequence gives 

\[
\begin{tikzcd}[row sep=1em]
0\arrow{r} & 0 \arrow{r} & 0 \arrow{r} & H^{0}\left(M'\right) \arrow{lld}\\ 
& \mathbb{Z}/3\mathbb{Z} \arrow{r} & H^{1}\left(M\right) \arrow{r} & H^{1}\left(M'\right) \arrow{lld}\\
& 0 \arrow{r} & H^{2}\left(M\right) \arrow{r} & H^{2}\left(M'\right) \\ & \end{tikzcd}\]and the second sequence gives

\[
\begin{tikzcd}[row sep=1em]
0\arrow{r} & H^{0}\left(N'\right) \arrow{r} & H^{0}\left(M'\right) \arrow{r} & 0 \arrow{lld}\\ 
& H^{1}\left(N'\right) \arrow{r} & H^{1}\left(M'\right) \arrow{r} & \mathbb{Z}/3\mathbb{Z} \arrow{lld}\\
& H^{2}\left(N'\right) \arrow{r} & H^{2}\left(M'\right) \arrow{r} & 0. \\ & \end{tikzcd}\]Putting these together gives the desired result.
\end{proof}

\subsection{The case of $S_{4}$ and $A_{4}$.}
\begin{lem}
Let $G=S_{4}$ or $A_{4}$. Then
\[
\mathrm{rk}_{2}Cl_{S}\left(K_{2}\right)+\left(t_{2}-t_{1}\right)-2\le\mathrm{rk}_{2}Cl_{S}\left(K_{1}\right)
\]
 and
\[
\mathrm{rk}_{2}Cl_{S}\left(K_{1}\right)\le\mathrm{rk}_{2}Cl_{S}\left(K_{2}\right)+\left(t_{2}-t_{1}\right)+1.
\]
\end{lem}
\begin{proof}
We consider the filtration \ref{S4sequence}. By Lemma \ref{lem:mu_p cohomology number fields}
we have $H^{1}\left(\mu_{2}\right)=\left(\mathbb{Z}/2\mathbb{Z}\right)^{2}$
and $H^{i}\left(\mu_{2}\right)=\mathbb{Z}/2\mathbb{Z}$ for $i=0,2$
. Furthermore $H^{0}\left(M\right)=H^{0}\left(N\right)=\mathbb{Z}/2\mathbb{Z}$.
Taking cohomology of the first sequence gives

\[
\begin{tikzcd}[row sep=1em]
0\arrow{r} & H^{0}\left(M_{1}\right) \arrow{r} & \mathbb{Z}/2\mathbb{Z} \arrow{r} & \mathbb{Z}/2\mathbb{Z} \arrow{lld}\\ 
& H^{1}\left(M_{1}\right) \arrow{r} & H^{1}\left(M\right) \arrow{r} & \left(\mathbb{Z}/2\mathbb{Z}\right)^{2} \arrow{lld}\\
& H^{2}\left(M_{1}\right) \arrow{r} & H^{2}\left(M\right) \arrow{r} & \mathbb{Z}/2\mathbb{Z} \\ & \end{tikzcd} \]and the second sequence gives

\[
\begin{tikzcd}[row sep=1em]
0\arrow{r} & \mathbb{Z}/2\mathbb{Z} \arrow{r} & H^{0}\left(M_{1}\right) \arrow{r} & 0 \arrow{lld}\\ 
& \left(\mathbb{Z}/2\mathbb{Z}\right)^{2} \arrow{r} & H^{1}\left(M_{1}\right) \arrow{r} & H^{1}\left(N'\right) \arrow{lld}\\
& \mathbb{Z}/2\mathbb{Z} \arrow{r} & H^{2}\left(M_{1}\right) \arrow{r} & H^{2}\left(N'\right). \\ & \end{tikzcd} \]From \ref{eq:S4directsum} we also get 
\[
2+h^{1}\left(N'\right)=h^{1}\left(N\right).
\]
Putting these together gives
\[
h^{1}\left(N\right)-2\le h^{1}\left(M\right)\le h^{1}\left(N\right)+1
\]
 and the result follows.
\end{proof}
Consider the case of $A_{4}$. In this case $-3\le u_{2}-u_{1}\le2$.
Furthermore we can compute $s_{1}-s_{2}$ as follows. If $L$ is totally
real then sig$\left(K_{1}\right)=\left(4,0\right)$ and sig$\left(K_{2}\right)=\left(3,0\right)$
so $s_{1}-s_{2}=1$ so 
\[
\mathrm{rk}_{2}Cl_{S}\left(K_{2}\right)-3+u_{2}-u_{1}\le\mathrm{rk}_{2}Cl_{S}\left(K_{1}\right)\le\mathrm{rk}_{2}Cl_{S}\left(K_{2}\right)+u_{2}-u_{1},
\]
and if $L$ is totally complex then sig$\left(K_{1}\right)=\left(0,2\right)$
and sig$\left(K_{2}\right)=\left(3,0\right)$ (since in this case
$K_{2}$ is Galois so must be totally real) so $s_{1}-s_{2}=-1$ and
\[
\mathrm{rk}_{2}Cl_{S}\left(K_{2}\right)-1+u_{2}-u_{1}\le\mathrm{rk}_{2}Cl_{S}\left(K_{1}\right)\le\mathrm{rk}_{2}Cl_{S}\left(K_{2}\right)+2+u_{2}-u_{1}.
\]
Combining this with Remark \ref{sclassgroup} when $L$ is totally
real gives 
\[
\mathrm{rk}_{2}Cl\left(K_{2}\right)-10\le\mathrm{rk}_{2}Cl\left(K_{1}\right)\le\mathrm{rk}_{2}Cl\left(K_{2}\right)+10
\]
 and when $L$ is totally complex gives 
\[
\mathrm{rk}_{2}Cl\left(K_{2}\right)-8\le\mathrm{rk}_{2}Cl\left(K_{1}\right)\le\mathrm{rk}_{2}Cl\left(K_{2}\right)+12.
\]

\subsection{The case of $D_{2l}$. }
\begin{lem}
If $G=D_{2l}$ and $K_{2}$ is disjoint from $\mathbb{Q}\left(\zeta_{l}\right)$
then 
\[
\mathrm{rk}_{l}Cl_{S}\left(K_{2}\right)+\left(t_{2}-t_{1}\right)\le\mathrm{rk}_{l}Cl_{S}\left(K_{1}\right)
\]
and
\[
\mathrm{rk}_{l}Cl_{S}\left(K_{1}\right)\le\frac{l-1}{2}\left(\mathrm{rk}_{l}Cl_{S}\left(K_{2}\right)+t_{2}+1\right)-t_{1}-1.
\]
\end{lem}
\begin{proof}
By Lemma \ref{lem:mu_p cohomology} $H^{1}\left(\mu_{l}\right)=\mathbb{Z}/l\mathbb{Z}$
and $H^{1}\left(\mu_{l}\right)=0$ for $i=0,2$. Furthermore $H^{0}\left(N\right)=H^{0}\left(M\right)=0$.
This implies that $H^{0}\left(M_{k}\right)=0$ for all $k$ and $H^{0}\left(N'\right)=0$.

We consider the filtration \ref{D2psequence1}. Taking cohomology
of the first sequence gives

\[
\begin{tikzcd}[row sep=1em]
0\arrow{r} & 0 \arrow{r} & 0 \arrow{r} & 0 \arrow{lld}\\ 
& H^{1}\left(M_{l-2}\right) \arrow{r} & H^{1}\left(M_{l-1}\right) \arrow{r} & \mathbb{Z}/l\mathbb{Z} \arrow{lld}\\
& H^{2}\left(M_{l-2}\right) \arrow{r} & H^{2}\left(M_{l-1}\right) \arrow{r} & 0. \\ & \end{tikzcd} \]Taking cohomology of the second sequence gives 

\[
\begin{tikzcd}[row sep=1em]
0\arrow{r} & 0 \arrow{r} & 0 \arrow{r} & 0 \arrow{lld}\\ 
& H^{1}\left(M_{l-3}\right) \arrow{r} & H^{1}\left(M_{l-2}\right) \arrow{r} & H^{1}\left(N'\right) \arrow{lld}\\
& H^{2}\left(M_{l-3}\right) \arrow{r} & H^{2}\left(M_{l-2}\right) \arrow{r} & H^{2}\left(N'\right). \\ & \end{tikzcd} \]And so on. The $k=1$ sequence gives 

\[
\begin{tikzcd}[row sep=1em]
0\arrow{r} & 0 \arrow{r} & 0 \arrow{r} & 0 \arrow{lld}\\ 
& \mathbb{Z}/l\mathbb{Z} \arrow{r} & H^{1}\left(M_{1}\right) \arrow{r} & H^{1}\left(N'\right) \arrow{lld}\\
& 0 \arrow{r} & H^{2}\left(M_{1}\right) \arrow{r} & H^{2}\left(N'\right). \\ & \end{tikzcd} \]By \ref{D2psequence2} we again have 
\[
h^{1}\left(N\right)=h^{1}\left(N'\right)+1.
\]
Putting these together gives the desired result.
\end{proof}
Compare this result with Proposition \ref{prop:Bolling}. Both the
upper and lower bounds are worse in our case.

\subsection{The case of $\mathbb{Z}/l\mathbb{Z}\rtimes\mathbb{Z}/r\mathbb{Z}$.}
\begin{lem}
If $G=\mathbb{Z}/l\mathbb{Z}\rtimes\mathbb{Z}/r\mathbb{Z}$ and $K_{2}$
is disjoint from $\mathbb{Q}\left(\zeta_{l}\right)$ then 
\begin{eqnarray*}
\frac{1}{r-1}\left(\mathrm{rk}_{l}Cl_{S}\left(K_{2}\right)+t_{2}\right)-t_{1}-1 & \le & \mathrm{rk}_{l}Cl_{S}\left(K_{1}\right)
\end{eqnarray*}
 and 
\begin{eqnarray*}
\mathrm{rk}_{l}Cl_{S}\left(K_{1}\right) & \le & \frac{\left(l-1\right)}{r}\left(\mathrm{rk}_{l}Cl_{S}\left(K_{2}\right)+t_{2}+1\right)+\frac{l-1}{r}-t_{1}.
\end{eqnarray*}
\end{lem}
\begin{proof}
By Lemma \ref{lem:mu_p cohomology number fields} we have $H^{1}\left(\mu_{l}\right)=\mathbb{Z}/l\mathbb{Z}$
and $H^{i}\left(\mu_{l}\right)=0$ for $i=0,2$ . Furthermore $H^{0}\left(M\right)=H^{0}\left(N\right)=0$.
This implies that $H^{0}\left(M_{k}\right)=0$ for all $k$ and $H^{0}\left(N_{k}\right)=0$
for all $k$. Additionally we note that by Lemma \ref{lem:euler char number fields}
we have $h^{2}\left(M_{k}\right)\le h^{1}\left(M_{k}\right)-h^{0}\left(M_{k}\right)$
and since $h^{0}\left(M_{k}\right)=0$ we get $h^{2}\left(M_{k}\right)\le h^{1}\left(M_{k}\right)$. 

Since cohomology commutes with direct sum we have
\[
h^{1}\left(N\right)=\sum_{k=0}^{r-1}h^{1}\left(R_{k}\right).
\]

Now we focus on the filtration \ref{semidirectsequence2}.  Taking
cohomology of the $k$th sequence when $k\neq1$ and $R_{k}\neq\mathbb{Z}/l\mathbb{Z}$
gives:

\[
\begin{tikzcd}[row sep=1em]
0\arrow{r} & 0 \arrow{r} & 0 \arrow{r} & 0 \arrow{lld}\\ 
& H^{1}\left(M_{k-1}\right) \arrow{r} & H^{1}\left(M_{k}\right) \arrow{r} & H^{1}\left(R_{k}\right) \arrow{lld}\\
& H^{2}\left(M_{k-1}\right) \arrow{r} & H^{2}\left(M_{k}\right) \arrow{r} & H^{2}\left(R_{k}\right) \\ & \end{tikzcd} \]and when $k=1$ gives:

\[
\begin{tikzcd}[row sep=1em]
0\arrow{r} & 0 \arrow{r} & 0 \arrow{r} & 0 \arrow{lld}\\ 
& \mathbb{Z}/l\mathbb{Z} \arrow{r} & H^{1}\left(M_{1}\right) \arrow{r} & H^{1}\left(R_{1}\right) \arrow{lld}\\
& 0 \arrow{r} & H^{2}\left(M_{1}\right) \arrow{r} & H^{2}\left(R_{1}\right) \\ & \end{tikzcd} \]and when $R_{k}=\mathbb{Z}/l\mathbb{Z}$gives: 

\[
\begin{tikzcd}[row sep=1em]
0\arrow{r} & 0 \arrow{r} & 0 \arrow{r} & 0 \arrow{lld}\\ 
& H^{1}\left(M_{k-1}\right) \arrow{r} & H^{1}\left(M_{k}\right) \arrow{r} & \mathbb{Z}/l\mathbb{Z} \arrow{lld}\\
& H^{2}\left(M_{k-1}\right) \arrow{r} & H^{2}\left(M_{k}\right) \arrow{r} & 0. \\ & \end{tikzcd} \]

Putting these together gives the desired result.
\end{proof}

\section*{Acknowledgements}

The author would like to thank Jacob Tsimerman for introducing him
to this problem and for many helpful discussions, and Franz Lemmermeyer
for providing helpful references.

\bibliographystyle{plain}
\bibliography{bibliography}

\textsc{~}

\textsc{\small{}Department of Mathematics, University of Toronto,
Canada}{\small \par}
\end{document}